\documentclass{amsart}
\usepackage{times}
\usepackage{amsthm}
\usepackage{amssymb}
\newtheorem{thrm}{Theorem}
\newtheorem{defi}[thrm]{Definition}
\newtheorem{prop}[thrm]{Proposition}
\newtheorem{lemma}[thrm]{Lemma}
\newtheorem{cor}[thrm]{Corollary}

\numberwithin{thrm}{section}

\newcommand{\la}{\mathfrak{g}}

\newcommand{\rtwo}{\mathbb{R}^2}
\newcommand{\cee}{\mathbb{C}}
\newcommand{\arr}{\mathbb{R}}
\newcommand{\gap}{\hspace{.25in}}
\newcommand{\ii }{\sqrt{-1}}

\begin{document}

\title{Homogeneous Solutions of Pluriclosed Flow on Closed Complex Surfaces}
\author{Jess Boling}

\begin{abstract}
Streets and Tian introduced a parabolic flow of pluriclosed metrics. We classify
the long time behavior of homogeneous solutions of this flow
on closed complex surfaces including minimal Hopf, Inoue, Kodaira, and
non-K\"ahler, properly elliptic surfaces. We also construct expanding soliton
solutions to the
flow on the universal covers of these surfaces by taking blowdown limits of
these homogeneous solutions.
\end{abstract}

\maketitle

\section{Introduction}

The Ricci flow $\partial_tg=-2Rc(g)$ is a well known tool for the study of Riemannian manifolds, but
is not well suited for Hermitian geometry because the Ricci tensor, in general,
is not Hermitian symmetric; under the Ricci flow an arbitrary Hermitian metric need not remain Hermitian for positive times. If the K\"ahler condition is imposed the Ricci tensor is Hermitian symmetric and the resulting flow is known as K\"ahler-Ricci flow. By the
$\partial\bar \partial$-lemma the K\"ahler-Ricci flow then reduces to a
parabolic Monge-Ampere equation for a single function. This reduction to a single scalar PDE makes analysis of the long time behavior and singularity formation of K\"ahler-Ricci flow tractable, and has
been used for example in \cite{cao1} to prove existence results for K\"ahler-Einstein metrics on
K\"ahler manifolds and in \cite{tia1} to give precise statements of the maximal
existence time for the flow in terms of cohomological data of the initial metric.

In the interest of studying general Hermitian manifolds, in particular to obtain classification results for non-K\"ahler complex manifolds, it is natural to consider geometric
flows constructed from curvature and torsion tensors associated to Hermitian compatible connections which agree with the Levi-Civita connection when the metric is K\"ahler. It is from this interest that we begin.

Let $(M^{2n},g,J)$ be a Hermitian manifold with complex structure $ J $ and
compatible Riemannian metric $ g $. The fundamental or K\"ahler 2-form
associated to the metric is $ \omega = g(J\cdot,\cdot) $ and we say the metric
is
\emph{pluriclosed} if \[ \partial\overline{\partial}\omega=0. \] In
\cite{st2,st1}, Streets and Tian introduce a parabolic flow of Hermitian
metrics 
which, when the metric is pluriclosed, is equivalent to the following
\begin{align} \label{pcf}
\frac{\partial \omega}{\partial
t}=\partial\partial^{\ast}\omega+\overline{\partial}\overline{\partial}^{\ast}
\omega+\frac{\ii }{2}\partial\overline{\partial}\log\det g.
\end{align}

We present the long time behavior of (\ref{pcf}) on all compact
Hermitian surfaces which are locally homogeneous. These include hyperelliptic,
Hopf, Inoue, Kodaira, and non-K\"ahler properly elliptic surfaces. We will also
construct expanding soliton solutions of the flow on the universal covers of
these spaces. Our theorem to this end will be:
\begin{thrm}\label{soli}
Let $g(\cdot)$ be a locally homogeneous solution of pluriclosed flow on a
compact
complex surface which exists on a maximal time interval $[0,T)$. 
If $T<\infty$ then the complex surface is rational or ruled. If $T=\infty$ and
the manifold is a Hopf surface, the evolving metric converges exponentially fast
to a canonical form unique up to homothety. Otherwise, there is a blowdown limit
\[\tilde g_\infty(t)=\underset{s\rightarrow\infty}\lim s^{-1}\tilde g(st)\] of
the
induced metric on the universal cover which is an expanding soliton in the sense
that $\tilde g(t) = t \tilde g(1)$ up to automorphism.
\end{thrm}
This is analogous to the construction of expanding Ricci solitons in \cite{lott2}. We also compute the time-rescaled Gromov-Hausdorff limits of our solutions and observe collapse to points, circles, and curves of high genus.
\begin{thrm}\label{grohau}
Let $g(\cdot)$ be a locally homogeneous solution of pluriclosed flow on a
compact
complex surface $(M,J)$ which exists on the interval $[0,\infty)$ and suppose
that $(M,J)$ is not a Hopf surface. Let $\hat{g}(t)=\frac{g(t)}{t}$.\\
1. If the surface is a torus, hyperelliptic, or Kodaira surface, then the family
$(M,\hat{g}(t))$ converges as $t\rightarrow\infty$ to a point in the
Gromov-Hausdorff sense.\\
2. If the surface is an Inoue surface, then the family $(M,\hat{g}(t))$
converges as $t\rightarrow\infty$ to a circle in the Gromov-Hausdorff sense and
moreover the length of this circle depends only on the complex structure of the
surface.\\
3. If the surface is a properly elliptic surface where the genus of the base
curve is at least 2, then the family $(M,\hat{g}(t))$ converges as
$t\rightarrow\infty$ to the base curve with a metric of constant curvature.\\
4. If the surface is of general type, then the family $(M,\hat{g}(t))$ converges
as $t\rightarrow\infty$ to a product of Kahler-Einstein metrics on $M$.
\end{thrm}

\emph{Remark:} We note that homogeneous solutions
of pluriclosed flow on Inoue and non-K\"ahler properly elliptic surfaces have
similar asymptotics and Gromov-Hausdorff limits as the example solutions to Chern-Ricci
flow on these surfaces considered in \cite{tw1}, and our arguments for the Gromov-Hausdorff limits in these cases are the same as in \cite{tw1}.

The organization of the paper is as follows. In Section 2 we will provide
background discussion and calculations for the homogeneous Hermitian geometries
considered throughout the paper. In section 3 we will analyze the long time
behavior of the flow for each of the geometries in section 2 and prove Theorem
\ref{grohau}. In Section 4 we complete the proof of Theorem \ref{soli} by performing
blowdown limits on the solutions of Section 3. We conclude the paper with a discussion of future work which would build off the results of this paper.

\emph{Acknowledgments:} The author would like to thank Jeff Streets for introducing him to this subject and for many helpful discussions.

\section{Background}

\subsection{Hermitian Connections and Curvature}

Let $(M^{2n},g,J)$ be a Hermitian manifold with fundamental 2-form $\omega$.
With a condition on the torsion, Gauduchon \cite{ga1} has shown that there is a
canonical 1-dimensional family of Hermitian connections $\nabla^\tau$ on $M$.
When $\tau=1$ we obtain the Chern connection $\nabla^c=\nabla^1$. It is defined
by
\[
g(\nabla^c_XY,Z) = g(\nabla^g_XY,Z)-\frac{1}{2}d\omega(JX,Y,Z)
\]
where $\nabla^g$ is the Levi-Civita connection of $g$. In local holomorphic
coordinates
the Chern connection $\nabla^{c} $ has coefficients
\[
\Gamma^{k}_{ij}=g^{\overline{l}k}g_{j\overline{l},i}
\]
where 
\[
g_{j\overline{l},i} = \frac{\partial}{\partial z^i} g_{j\overline{l}}.
\]
If $R^c(X,Y)Z=([\nabla^{c}_X,\nabla^{c}_Y]-\nabla^{c}_{[X,Y]})Z$ is the $(3,1)$
curvature tensor of the Chern connection, its Ricci form is defined by
\[
\rho^{c}_\omega = \frac{1}{2} \underset{i}\Sigma R^c(X,Y,Je_i,e_i).
\]
Here $e_i$ is an orthonormal basis of the (real) tangent space. In local
holomorphic coordinates the Ricci form of the Chern connection is given by
\[
\rho^{c}_\omega = -\frac{\ii }{2}\partial\overline{\partial}\log\det g.
\]
In a more invariant form, if $\omega$ and $\omega_0$ are two hermitian metrics,
their Chern-Ricci forms are related by
\[
\rho^{c}_\omega = \rho^{c}_{\omega_0} - \ii \partial
\overline{\partial}\log\frac{\omega^n}{\omega^{n}_{0}},
\]
where $\omega^n=\omega\wedge\omega\wedge\ldots\wedge\omega$.
This formula is useful for computing with invariant metrics. For example, it
is immediate that $\rho^c_\omega=\rho^c_{\omega_0}$ whenever $\omega$ and
$\omega_0$ have proportionally constant volume forms. 

The Bismut connection is defined by
\[
g(\nabla^{b}_{X}Y,Z) = g(\nabla^{g}_{X}Y,Z) + \frac{1}{2}d\omega(JX,JY,JZ)
\]
and is the connection corresponding to $\tau=-1$ in the family of Gauduchon. As
with the Chern connection, there is an associated Bismut-Ricci form 
\[
\rho^{b}_\omega =  \frac{1}{2} \underset{i}\Sigma R^b(X,Y,Je_i,e_i).
\]
The Bismut-Ricci and Chern-Ricci forms are related by
\[
\rho^{b}_\omega =\rho^{c}_\omega -dd^{*}\omega
\]
and so we obtain
\begin{lemma}
Pluriclosed flow is given by
\[
\frac{d\omega}{dt}=\partial\partial^{\ast}\omega+\overline{
\partial\partial}^{\ast}\omega+\frac{\ii }{2}\partial\overline{\partial}
\log\det g =-(\rho^{b}_{\omega})^{(1,1)}.
\]
\end{lemma}
\subsection{Homogeneous Hermitian Geometries}
In this section we will give a short summary of homogeneous Hermitian geometry in complex dimension two. We will begin with a few notions from homogeneous Riemannian geometry.
\begin{defi}
A Riemannian manifold $(M^{n},g)$ is \emph{locally homogeneous} if for any two
points
$x,y\in M$ there exists an isometry $\theta:U\rightarrow V$ between open
neighborhoods $x\in U$ and $y\in V$ such that $\theta(x)=y$. $(M,g)$ is
\emph{globally homogeneous} if the locally defined isometries $\theta$ are
defined on all of $M$. 
\end{defi}
Given that the universal cover of any complete, locally homogeneous Riemannian
manifold is globally homogeneous, in order to classify complete, locally
homogeneous Riemannian manifolds it is enough to classify the simply connected,
complete, homogeneous Riemannian manifolds together with their co-compact
lattices. With a globally homogeneous Riemannian manifold the isometry group
$Iso(M,g)$
acts transitively on $M$. Then the isotropy groups
\[
G_x=\{\theta \in Iso(M,g)\ \vline \ \theta(x)=x \}
\]
are all conjugate, isomorphic to closed subgroups of $O(n)$, and $M$ is
diffeomorphic to the quotient $G/G_x$. A first step toward the classification of such spaces,
then, is to find closed subgroups of $O(n)$ of dimension $k$ and embeddings of
these subgroups into unimodular Lie groups of dimension $n+k$. That the same manifold $M$ can arise from different quotients in this way is addressed with the following definition.
\begin{defi}
A (minimal) \emph{model geometry} is\\
1. A complete, simply connected, homogeneous Riemannian manifold $(M,g)$ where
the metric $g$ is the pullback of a metric on some compact manifold whose
universal cover is $M$, together with\\
2. A closed subgroup $G$ of $Iso(M,g)$ acting transitively on $M$ such that $G$
is minimal with this property.
\end{defi}
A list of the four dimensional model geometries can be found in \cite{ijl} in
the context of homogeneous Ricci flows and is summarized below. Note that we
have omitted the groups $G$. 
\begin{prop}\label{model}
The four-dimensional model geometries are either\\
1. Products of two-dimensional model geometries $S^2,\rtwo$ and $H^2$, \\
2. $S^4,\cee P^2,H^3\times\arr,\cee H^2,H^4$, or\\
3. A simply connected 4-dimensional Lie group which has a co-compact lattice.\\
The metrics in the first and second cases are products of canonical, well known
Einstein metrics, while the metrics on each of the Lie groups are left
invariant.
\end{prop}
Here $H^n$ is hyperbolic space of dimension $n$, $\cee P^n$ is complex
projective space with a Fubini-Study metric, and $\cee H^n$ is complex
hyperbolic space: an open ball in $\cee^n$ with the Bergman metric.
\begin{defi}
A complex structure $J$ is \emph{compatible} with a model geometry $(M,g,G)$ if
$G$ acts by holomorphic isometries. If $J$ is compatible we say that $(M,g,J)$
is \emph{homogeneous Hermitian}.
\end{defi}
Wall \cite{wa1} classified the (integrable) complex structures up to isomorphism
and conjugation which are compatible with a given 4-dimensional model geometry.
\begin{prop}
The 4-dimensional model geometries which admit a compatible integrable
complex structure are either\\
1. Products of the two-dimensional model geometries $S^2=\cee P^1$, $\arr
^2=\cee$,
and $H^2=\cee H^1$, with canonical product complex structures,\\
2. $\cee P^2$ and $\cee H^2$, with canonical complex structures, or\\
3. A 4-dimensional Lie group with a left invariant integrable complex structure
\end{prop}
\begin{proof}
We will give a brief sketch; for details see \cite{wa1}. Let $J$ be a compatible
integrable complex structure on a four-dimensional model geometry $(M,g,G)$.
First, note that $J$ must commute with the action of the isotropy group on the
tangent space at each point. For the model geometries $S^4$, $H^3\times\arr$,
and $H^4$ where the isotropy group contains a copy of $SO(3)$, there can be no
such commuting $J$. The remaining cases are either Lie groups, $\cee P^2$, $\cee
H^2$, or products of two-dimensional models. The standard complex structures on
$\cee P^2$, $\cee H^2$, and the product cases are compatible with the model
geometry structure and are the unique compatible complex structures up to
automorphism and conjugation. The Lie group cases are treated as follows.

A compatible complex structure on a Lie group $G$ is determined by its action on
the Lie algebra $\la$ of left invariant vector fields. Note that $J$ is
integrable if, and only if, the $\ii $-eigenspace of the complexification
$J^\cee:\la^\cee \rightarrow\la^\cee$ is a Lie subalgebra of $\la^\cee$. Let
$e_i$ be a basis of $\la$ and let $c_{ij}^k$ be the structure constants of $\la$
with respect
to such a basis, so that
\[
[e_i,e_j]=c_{ij}^ke_k.
\]
Because there are only two 2-dimensional complex Lie algebras up to
isomorphism, if $J$ is integrable there is a basis $Z_1,Z_2$ of the
$\ii $-eigenspace of $J^\cee$ where either $[Z_1,Z_2]=0$ or
$[Z_1,Z_2]=Z_2$. Writing $Z_i=a_i^je_j$ for some $a_i^j\in\cee$, the
integrability condition becomes
\begin{align}\label{scc}
a_1^ia_2^jc_{ij}^k=\epsilon a_2^k
\end{align}
where $\epsilon=0$ or $1$ depending on whether or not the $\ii $-eigenspace is
abelian. Once a solution $a_i^j$ to the above equations is known, if we let
\[
X_1=(\Re a_1^j)e_j, X_2=-(\Im a_1^j)e_j, X_3=(\Re a_2^j)e_j, X_4=-(\Im
a_2^j)e_j,
\] 
then an integrable complex structure is given by $JX_1=X_2$ and $JX_3=X_4$. As
there are possibly many solutions to the above system, automorphisms of $\la$
substantially reduce the number of cases that one needs to consider. Doing this
on a case by case basis for each of the Lie algebras in Proposition
(\ref{model}) completes the proof and gives the list of complex structures.
\end{proof}
We now list the complex structures and corresponding groups of the previous
proposition. In each case we will use a left invariant basis $Z_i$ of $T^{1,0}M$
and give the Lie brackets with respect to this basis; the complex structure is
$J^\cee Z_i=\ii Z_i$. When there is more than one compatible complex structure
we list a family of Lie brackets with respect to such a basis.
\begin{enumerate}
\item On $\arr^4$ there is a unique compatible complex structure. Quotients by a
cocompact lattice in this case give complex tori.
\item On $\tilde E(2)\times\arr$ there is a unique compatible complex structure.
Here $\tilde E(2)$ is the universal covering group of the rigid motions of the
Euclidean plane. The non-vanishing Lie brackets with respect to a $T^{1,0}$
frame are
\[[Z_1,Z_2]=Z_1\gap [Z_1,\overline{Z_2}]=-Z_1\]
and compact quotients in this case are hyperelliptic surfaces.
\item On $\arr\times SU_2$ there is a one parameter family of compatible complex
structures (parameterized by $\alpha\in\arr$) with brackets
\[[Z_1,Z_2]=Z_2\gap[Z_1,\overline{Z_2}]=-\overline{Z_2}\]
\[[Z_2,\overline{Z_2}]=(\ii \alpha-1)Z_1+(\ii \alpha+1)\overline{Z_1}.\]
Here the compact quotients give Hopf surfaces.
\item On $\tilde{SL}_2(\arr)\times\arr$ there is a one parameter family of
compatible complex structures with brackets
\[[Z_1,Z_2]=\ii Z_1\gap[Z_1,\overline{Z_2}]=\ii Z_1\]
\[[Z_1,\overline{Z_1}]=(\ii -\alpha)Z_2+(\ii +\alpha)\overline{Z_2}.\]
Quotients by a cocompact lattice in this case give non-K\"ahler, properly
elliptic surfaces.
\item On $Nil^3\times\arr$ there is a unique compatible complex structure whose
brackets are
\[[Z_1,\overline{Z_1}]=\ii (Z_2+\overline{Z_2}).\]
The quotients in this case form Kodaira surfaces.
\item There is a semi-direct product $Nil^3\rtimes \arr$ with two compatible
complex structures given by
\[[Z_1,Z_2]=\epsilon Z_1\gap[Z_1,\overline{Z_2}]=-\epsilon Z_1,\]
\[[Z_1,\overline{Z_1}]=-\epsilon\ii (Z_2+\overline{Z_2}).\]
where $\epsilon=\pm 1$.
Compact quotients in this case are Kodaira surfaces.
\item There is a family of solvable Lie groups with complex structures
\[[Z_1,Z_2]=\lambda Z_1\gap[Z_1,\overline{Z_2}]=-\lambda Z_1,\]
\[[Z_2,\overline{Z_2}]=2a\ii (Z_2+\overline{Z_2}),\]
where $\lambda=-b+\ii a$ is a complex number. These do not always have a
cocompact lattice; when such a lattice exists the quotient forms an Inoue
surface.
\item The solvable Lie group $Sol_1^4$ has two compatible complex structures.
The first is given by
\[[Z_1,Z_2]=-Z_2\gap[Z_1,\overline{Z_2}]=-Z_2,\]
\[[Z_1,\overline{Z_1}]=\overline{Z_1}-Z_1,\]
and we will call $Sol_1^4$ with this complex structure $Sol_1$.
The second is given by
\[[Z_1,Z_2]=-Z_2\gap[Z_1,\overline{Z_2}]=-Z_2,\]
\[[Z_1,\overline{Z_1}]=\overline{Z_1}-Z_1+Z_2-\overline{Z_2},\]
which we will call $Sol_1'$. The quotients in each these cases form Inoue surfaces.
\end{enumerate}
\subsection{The Bismut-Ricci Form of a Left Invariant Metric}
We next record a basic lemma concerning the left invariant Hermitian metrics that we will consider on the previous Lie groups.
\begin{lemma}
Let $(M^{4},J)$ be a 4-dimensional Lie group with a left invariant complex
structure. With a basis $Z_i$ of left invariant $T^{1,0}M$ vector fields, any
Hermitian metric $g$ is determined by complex valued functions
$x=g(Z_1,\overline{Z_1})$, $y=g(Z_2,\overline{Z_2})$, and
$z=g(Z_1,\overline{Z_2})$ satisfying $x,y>0$ and $xy-|z|^2>0$. If $\zeta^i$ is
the dual basis to the $Z_i$, the K\"ahler form
$\omega(\cdot,\cdot)=g(J\cdot,\cdot)$ is given by
\[\omega=\ii
(x\zeta^{1\overline{1}}+y\zeta^{2\overline{2}}+z\zeta^{1\overline{2}}+\overline{
z}\zeta^{2\overline{1}}),\]
where $\zeta^{i\overline{j}}=\zeta^i\wedge\zeta^{\overline{j}}$ is the wedge
product.
The metric $g$ is left invariant if and only if $x$, $y$, and $z$ are constant
on $M$.
\end{lemma}
A formula for the Ricci form associated to each of the canonical connections $\nabla^\tau$ for a left invariant Hermitian metric was computed in \cite{ve1}, the following proposition concerns the special case of the Bismut connection.
\begin{prop}\label{vez}\cite{ve1}
Let $(M^{2n},g,J)$ be a Lie group with a left invariant Hermitian structure.
Then the Bismut-Ricci form can be written as $\rho^b=d\eta$, where
\[\eta=\eta_i\zeta^i+\eta_{\overline{i}}\zeta^{\overline{i}}\]
\[\eta_i=\ii c_{ij}^j-\ii
g^{\overline{j}k}c_{k\overline{j}}^{\overline{l}}g_{i\overline{l}}\]
and $c_{ij}^k, c_{i\overline{j}}^{\overline{k}}$ are the structure constants of
the Lie algebra with respect to the $Z_i,\overline{Z_i}$
\end{prop}
\section{Pluriclosed Flow on the Model Geometries}
We are now ready to compute the pluriclosed flow equations
\[\frac{d\omega}{dt}=-(\rho^b_\omega)^{(1,1)}\] 
for the homogeneous Hermitian metrics of the previous section and to prove
Theorem \ref{grohau}. Recall that our metrics have the form
\[\omega=\ii
(x\zeta^{1\overline{1}}+y\zeta^{2\overline{2}}+z\zeta^{1\overline{2}}+\overline{
z}\zeta^{2\overline{1}})\]
with respect to the above $T^{1,0}M$ frames.
\subsection{Hyperelliptic Surfaces}
\begin{lemma} Let $\omega$ be a left invariant Hermitian metric on $\tilde
E(2)\times\mathbb{R}$ and let $\zeta^i$ be a $(T^{1,0})^*M$ frame satisfying
\[d\zeta^1=-\zeta^{12}+\zeta^{1\overline{2}}\gap d\zeta^2=0.\]
Then
\[\eta_1=\ii \frac{zx}{xy-|z|^2},\]
\[\rho^b=\ii \frac{zx}{xy-|z|^2}(-\zeta^{12}+\zeta^{1\overline{2}})+conjugates.\]
\end{lemma}
\begin{cor}
Pluriclosed flow for a left invariant metric on $\tilde E(2)\times\mathbb{R}$
satisfies
\[x'=y'=0\]
\[z'=-\frac{xz}{xy-|z|^2}.\]
Where $x',y',z'$ denote time derivatives of $x,y,z$. In particular,
$|z|=O(e^{-Ct})$ for some constant $C$ depending on the initial condition.
\end{cor}
\begin{cor}
Under pluriclosed flow a homogeneous Hermitian metric $g$ on a hyperelliptic
surface converges exponentially fast in the $C^\infty$ topology to a flat
K\"ahler metric. Under the family of metrics $\frac{g(t)}{t}$, a hyperelliptic
surface converges to a point in the Gromov-Hausdorff sense. 
\end{cor}
\emph{Remark}: Notice that $\omega$ is a flat K\"ahler metric if and only if
$z=0$. This is therefore an example of pluriclosed flow taking non-K\"ahler initial data to a K\"ahler metric.
\begin{proof}
A direct calculation from Proposition gives the Bismut-Ricci form and the resulting pluriclosed ODE. From here, we compute
\[(|z|^2)'=-2\frac{x_0|z|^2}{x_0y_0-|z|^2}\leq-2\frac{|z|^2}{y_0},\]
and so
\[|z|\leq|z_0|e^{-\frac{1}{y_0}t}.\]
\end{proof}
\subsection{Hopf Surfaces}
\begin{lemma}
Let $\omega$ be a left invariant Hermitian metric on $\arr\times SU_2$. With respect to a frame satisfying
\[d\zeta^1=(1-\ii \alpha)\zeta^{2\overline{2}}\gap
d\zeta^2=-\zeta^{12}-\zeta^{2\overline{1}}\]
we have
\[\eta_1=\frac{\alpha x^2+\ii(xy-x^2-2|z|^2)}{xy-|z|^2}\]
\[\eta_2=\frac{\alpha x\overline{z}-\ii \overline{z} (x+y)}{xy-|z|^2},\]
and so
\[\rho^b=(1-\ii \alpha)\frac{\alpha x^2+\ii
(xy-x^2-2|z|^2)}{xy-|z|^2}\zeta^{2\overline{2}}+\]
\[\frac{-\alpha x\overline{z}+\ii \overline{z}
(x+y)}{xy-|z|^2}(\zeta^{12}+\zeta^{2\overline{1}})+conjugates.\]
\end{lemma}
\begin{cor}
Pluriclosed flow for a left invariant metric on $\arr\times SU_2$ is given by
\[x'=0\]
\[y'=2\frac{x((\alpha^2+1)x-y)+2|z|^2}{xy-|z|^2}\]
\[z'=\frac{\alpha \ii xz-z(x+y)}{xy-|z|^2}\]
in particular, $|z|=O(e^{-\frac{1}{x_0}t})$ and $y\rightarrow(1+\alpha^2)x_0$.
\end{cor}
\begin{cor}
Under pluriclosed flow, a locally homogeneous Hermitian metric on a Hopf surface
converges in the $C^\infty$ topology to a metric which is independent of the initial condition and is unique up to homothety.
\end{cor}
\begin{proof}
We compute
\[(|z|^2)'=-2\frac{|z|^2(x_0+y)}{x_0y-|z|^2}\]
and so
\[(|z|^2)'\leq -2\frac{|z|^2}{x_0}.\]
Thus
\[|z|^2\leq|z_0|^2e^{-\frac{2}{x_0}t}.\]
Note that $y$ is increasing whenever $y < (1+\alpha^2)x_0$ and if $y$ is
nondecreasing then
\[y \leq (\alpha^2+1)x_0+2\frac{|z_0|^2}{x_0}e^{-\frac{2}{x_0}t}.\]
Therefore $y\rightarrow(1+\alpha^2)x_0$.
\end{proof}
\emph{Remark:} A homogeneous Hopf surface is a compact complex surface whose universal cover is $\cee^2\backslash\{0\}$ and which has a finite index subgroup of the fundamental group generated by the map $\theta(z,w)=(az,bw)$, where $a,b\in \cee^*$ and $|a|=|b|<1$. One can identify $\arr\times SU_2$ with $\cee^2\backslash\{0\}$ so that the induced complex structure is left invariant and the map $\theta$ is given by left multiplication by the element $(|a|,id)\in \arr\times SU_2$. The parameter $\alpha$ then encodes information on the angles $\arg a$ and $\arg b$ of $a$ and $b$. Because $|a|=|b|$, the metric \[\omega=\ii \frac{1}{r^2}\partial\bar{\partial}r^2\] descends to the quotient Hopf surface. As shown by Ivanov and Gauduchon \cite{ivga}, up to homothety this can be the only non-K\"ahler fixed point of the flow on compact complex surfaces.
\subsection{Non-K\"ahler, Properly Elliptic Surfaces}
\begin{lemma}
Let $\omega$ be a left invariant Hermitian metric on
$\tilde{SL_2\arr}\times\arr$. With the frame given by
\[d\zeta^1=-\ii (\zeta^{12}+\zeta^{1\overline{2}})\gap
d\zeta^2=(\alpha-\ii)\zeta^{1\overline{1}},\]
we compute
\[\eta_1=\frac{z(y-x)}{xy-|z|^2}+\ii \frac{-\alpha yz}{xy-|z|^2}\]
\[\eta_2=\frac{xy+y^2-2|z|^2}{xy-|z|^2}+\ii \frac{-\alpha y^2}{xy-|z|^2},\]
and therefore
\[\rho^b=(\frac{-\alpha yz+\ii
z(x-y)}{xy-|z|^2})(\zeta^{12}+\zeta^{1\overline{2}})+\]
\[(\alpha-\ii)(\frac{xy+y^2-2|z|^2-\ii\alpha
y^2}{xy-|z|^2})\zeta^{1\overline{1}}+conjugates.\]
\end{lemma}
\begin{cor}
Pluriclosed flow of a left invariant Hermitian metric on
$\tilde{SL_2\arr}\times\arr$ is given by
\[x'=2(1+\frac{(1+\alpha^2)y^2-|z|^2}{xy-|z|^2})\gap y'=0\]
\[z'=\frac{-\ii \alpha yz+z(y-x)}{xy-|z|^2}.\]
In particular, $z=O(e^{-Ct})$ and $x\sim 2t$.
\end{cor}
\begin{cor}
If $\omega(t)$ is a locally homogeneous solution to pluriclosed flow on a
non-K\"ahler properly elliptic surface, then under the family of metrics
$\frac{\omega(t)}{t}$ the surface converges to the base curve in the
Gromov-Hausdorff sense.
\end{cor}
\begin{proof}
We compute that $xy-|z|^2$ is increasing since
\[2xy^2+2(1+\alpha^2)y^3-4|z|^2y+2|z|^2(x-y)> 0\]
whenever $xy-|z|^2>0$.
Then note that $x'\geq 4$ whenever $x\leq y$, so for any $\delta>0$
\[(|z|^2)'\leq|z|^22(\delta-\frac{1}{y_0})\]
for sufficiently large $t$. Thus
\[|z|\leq Ce^{(\delta-\frac{1}{y_0})t}\]
for some constant $C$. $x\sim 2t$ is then immediate.

If $\pi:\Gamma\backslash\tilde{SL_2\arr}\times\arr\rightarrow \Sigma$ is the
projection of a non-K\"ahler properly elliptic surface to the base curve
$\Sigma$, the fibers are the leaves of the distribution spanned by the real and
imaginary parts of $Z_2$. Moreover, there is a unique metric $\omega_\Sigma$ on
$\Sigma$ such that $\pi^*\omega_\Sigma=2\ii \zeta^{1\bar 1}$. Now, if
$f:\Sigma\rightarrow\Gamma\backslash\tilde{SL_2\arr}\times\arr$ is any function
(not necessarily continuous) such that $\pi\circ f=id$ then, for any
$\epsilon>0$, $\pi$ and $f$ are $\epsilon$-Gromov-Hausdorff approximations with
respect to the metrics $\frac{\omega(t)}{t}$ and $\omega_\Sigma$ as long as $t$
is sufficiently large.
\end{proof}
\subsection{Kodaira Surfaces}
\begin{lemma}
Let $\omega$ be a left invariant metric on $Nil^3\times\arr$. With a frame satisfying
\[d\zeta^1=0\gap d\zeta^2=-\ii\zeta^{1\overline{1}},\]
we compute
\[\eta_2=\frac{y^2}{xy-|z|^2},\]
and so
\[\rho^b=-\ii\frac{y^2}{xy-|z|^2}\zeta^{1\overline{1}}+conjugate.\]
\end{lemma}
\begin{cor}
Pluriclosed flow of a left invariant metric on $Nil^3\times\arr$ is given by
\[x'=2\frac{y^2}{xy-|z|^2}\]
\[y'=z'=0.\]
In particular
\[\frac{1}{2}x^2y_0-x|z_0|^2=2y_0^2t+\frac{1}{2}x_0^2y_0-x_0|z_0|^2\]
for all $t$.
\end{cor}
\begin{cor}
Let $g(t)$ be a homogeneous solution of pluriclosed flow on a primary Kodaira
surface. Then under the metrics $\frac{g(t)}{t}$ the surface converges to a
point in the Gromov-Hausdorff sense.
\end{cor}
\begin{proof}
This is a direct computation.
\end{proof}
\begin{lemma}
Let $\omega$ be a left invariant Hermitian metric on $Nil\rtimes\arr$. With a frame satisfying
\[d\zeta^1=\epsilon(-\zeta^{12}+\zeta^{1\overline{2}})\gap
d\zeta^2=\epsilon\ii\zeta^{1\overline{1}},\]
we compute
\[\eta_1=\epsilon(-\frac{y z}{xy-|z|^2}+\ii \frac{xz}{xy-|z|^2}),\]
\[\eta_2=\epsilon(-\frac{y^2}{xy-|z|^2}+\ii \frac{2|z|^2-xy}{xy-|z|^2}),\]
and therefore
\[\rho^b=\frac{-yz+\ii xz}{xy-|z|^2}(-\zeta^{12}+\zeta^{1\overline{2}})+\]
\[\frac{xy-2|z|^2-\ii y^2}{xy-|z|^2}\zeta^{1\overline{1}}+conjugates.\]
\end{lemma}
\begin{cor}
Pluriclosed flow of a left invariant metric on $Nil\rtimes\arr$ is given by
\[x'=2\frac{y^2}{xy-|z|^2}\gap y'=0\]
\[z'=-\frac{(x+\ii y)z}{xy-|z|^2}.\]
In particular $x\sim 2\sqrt{y_0 t}$ and $|z|=O(e^{-Ct})$.
\end{cor}
\begin{cor}
Let $g(t)$ be a homogeneous solution of pluriclosed flow on a secondary Kodaira
surface. Then under the metrics $\frac{g(t)}{t}$ the surface converges to a
point in the Gromov-Hausdorff sense.
\end{cor}
\begin{proof}
Similar to the previous cases, we compute
\[(|z|^2)'\leq-\frac{|z|^2}{y_0},\]
and so $|z|=O(e^{-Ct})$. Then note that
\[\frac{2y_0}{x}\leq x' \leq \frac{2y_0^2}{xy_0-|z_0|^2}.\]
\end{proof}
\subsection{Inoue Surfaces}
\begin{lemma}
Let $\omega$ be a left invariant Hermitian metric on the solvable Lie group with frame satisfying
\[d\zeta^1=-\lambda\zeta^{12}+\lambda\zeta^{1\overline{2}}\gap d\zeta^2=-2a\ii
\zeta^{2\overline{2}}.\]
Then we compute
\[\eta_1=\frac{2azx+\ii \bar{\lambda}zx}{xy-|z|^2},\]
\[\eta_2=\frac{\ii (\lambda+\bar{\lambda})|z|^2+(2a-\ii \lambda)xy}{xy-|z|^2}.\]
Therefore the Bismut-Ricci form is
\[\rho^b=\frac{2azx+\ii
\bar{\lambda}zx}{xy-|z|^2}(-\lambda\zeta^{12}+\lambda\zeta^{1\overline{2}})+\]
\[(\frac{2a(\lambda+\bar{\lambda})|z|^2+(-4a^2\ii
-2a\lambda)xy}{xy-|z|^2})(\zeta^{2\overline{2}})+conjugates.\]
\end{lemma}
\begin{cor}
Pluriclosed flow of a left invariant metric on the solvable family is given by
\[x'=0,\]
\[y'=12a^2(1+\frac{|z|^2}{xy-|z|^2}),\]
\[z'=-(3a^2+b^2+2ab\ii)\frac{xz}{xy-|z|^2}.\]
In particular, $y\sim 12a^2t$ and $|z|$ is bounded.
\end{cor}
\begin{proof}
We compute
\[(|z|^2)'=-(3a^2+b^2)\frac{x_0|z|^2}{x_0y-|z|^2},\]
which shows that $|z|$ is bounded. It is then immediate that $y(t)/t\rightarrow
12a^2$.
\end{proof}
\begin{cor}
Let $\omega(t)$ be a locally homogeneous solution to pluriclosed flow on an
Inoue surface of type $S_A$. Then under the family of metrics $\frac{g(t)}{t}$,
$S_A$ converges to a circle of length $\sqrt{6}|a|$ in the Gromov-Hausdorff
sense. 
\end{cor}
\begin{proof} 
We recall the construction of an Inoue surface of type $S_A$. Choose a matrix
$A\in SL_3\mathbb{Z}$ with eigenvalues $\alpha$, $\bar{\alpha}$, and
$c=|\alpha|^{-2}$, where $\alpha\neq\bar{\alpha}$ and $|\alpha|\neq 1$. Write
$\alpha=|\alpha|e^{i\theta}$. Consider the solvable Lie group $G$ of matrices of
the form
\[\begin{pmatrix}\alpha^s & 0 & w\\0 & c^s & r\\0 & 0 & 1\end{pmatrix},\] where
$r,s\in\arr$, $w\in\cee$. This group has Lie brackets given by
$[X_4,X_1]=aX_1+bX_2$, $[X_4,X_2]=-bX_1+aX_2$, and $[X_4,X_3]=-2aX_3$, where
$a=\log |\alpha|$ and $b=\theta$. There is a natural identification of $G$ with
$\cee\times\cee H^1$ given by sending an element as before to $(w,r+c^s\ii)$.
Under such an identification, left multiplication is a biholomorphism. Let
$(a_1,a_2,a_3)^T$ be an eigenvector of $\alpha$ and $(c_1,c_2,c_3)^T$ be a real
eigenvector of $c$. Consider the lattice $\Gamma$ generated by the elements
\[g_0=\begin{pmatrix}
\alpha & 0 & 0\\0 & c & 0\\0 & 0 & 1
\end{pmatrix}\gap 
g_i=\begin{pmatrix}
1 & 0 & a_i \\ 0 & 1 & c_i \\ 0 & 0 & 1
\end{pmatrix}.\]
As shown by Inoue \cite{ino}, the quotient $\Gamma\backslash G=S_A$ forms a
compact complex surface of class $VII_0$ with the property that it has no
complex curves. Moreover, $S_A$ is a $T^3$ bundle over $S^1$, where the
projection $\pi:S_A\to S^1$ is given by mapping (the equivalence classes of)
$(w,r+c^s\ii)$ to $c^s$; this projection is induced by the natural projection of
$\cee\times\cee H^1$ to the imaginary axis of the second factor.

Consider the closed curve $\gamma:S^1\to S_A$ given by
$\gamma(s)=(0,c^s\ii)$. Note that this is well defined as a map of equivalence
classes and $\pi\circ\gamma=id$. Let $\omega(t)$ be a left invariant solution to
pluriclosed flow on $G$. With respect to $\omega(t)$, $\gamma$ has length
$L_\omega(\gamma)=\sqrt{\frac{y(t)}{2}}$. Therefore, with respect to the metrics
$\frac{\omega(t)}{t}$, the length of $\gamma$ approaches $\sqrt{6}|a|$ as
$t\to\infty$. Because there are no curves in $S_A$, the real and imaginary parts
of $Z_1$ form an integrable distribution whose leaves are dense in each $T^3$
fiber. Now, because the length of $Z_1$ with respect to $\omega(t)$ is fixed in
time, for any $\epsilon>0$ the diameter of each $T^3$ fiber with respect to the
metric $\frac{\omega(t)}{t}$ is less than $\epsilon$ for $t$ sufficiently large.
Therefore $\gamma$ and $\pi$ are $\epsilon$-Gromov-Hausdorff approximations
between the circle of length $\sqrt{6}|a|$ and $S_A$ with the metric
$\frac{\omega(t)}{t}$ for $t$ 
sufficiently large.
\end{proof}
\begin{lemma}
Let $\omega$ be a left invariant Hermitian metric on $Sol_1$ with frame given by
\[d\zeta^1=\zeta^{1\bar 1}\gap d\zeta^2=\zeta^{12}+\zeta^{1\bar 2}.\]
Then
\[\eta_1=-\ii \frac{2xy-|z|^2-z^2}{xy-|z|^2},\]
\[\eta_2=-\ii \frac{y(\bar{z}-z)}{xy-|z|^2},\]
and so
\[\rho^b_{1\bar 1}=-\ii \frac{4xy-(z+\bar z)^2}{xy-|z|^2},\]
\[\rho^b_{2\bar 2}=0,\]
\[\rho^b_{1\bar 2}=-\ii \frac{y(\bar z - z)}{xy-|z|^2}\]
\end{lemma}
\begin{cor}
Pluriclosed flow of a left invariant metric on $Sol_1$ is given by
\[x'=\frac{4xy-(z+\bar z)^2}{xy-|z|^2}\]
\[y'=0\]
\[z'=\frac{y(\bar z-z)}{xy-|z|^2}.\]
In particular, $x\sim 4t$ and $z$ is bounded.
\end{cor}
\begin{proof}
We compute
\[(|z|^2)'=\frac{y(z-\bar z)^2}{xy-|z|^2}\leq 0,\]
which shows that $|z|$ is bounded. Then, noting that $x'\geq 4$, we conclude the
result.
\end{proof}
\begin{lemma}
Let $\omega$ be a left invariant Hermitian metric on $Sol_1'$ with frame satisfying
\[d\zeta^1=\zeta^{1\bar 1}\gap d\zeta^2=-\zeta^{1\bar 1}+\zeta^{1\bar
2}+\zeta^{12}.\]
Then
\[\eta_1=-\ii \frac{2xy-yz-|z|^2-z^2}{xy-|z|^2},\]
\[\eta_2=-\ii \frac{y(\bar z - z)-y^2}{xy-|z|^2},\]
and so
\[\rho^b_{1\bar 1}=-\ii \frac{4xy-y(z+\bar z)-(z+\bar z)^2+2y^2}{xy-|z|^2},\]
\[\rho^b_{2\bar 2}=0\]
\[\rho^b_{1\bar 2}=-\ii \frac{y(\bar z + z)-y^2}{xy-|z|^2}\]
\end{lemma}
\begin{cor}
Pluriclosed flow of a left invariant metric on $Sol_1'$ is given by
\[x'=\frac{4xy-y(z + \bar z)-(z + \bar z)^2+2y^2}{xy-|z|^2}\gap y'=0,\]
\[z'=\frac{y(\bar z - z)-y^2}{xy-|z|^2}.\]
In particular, $x\sim 4t$ and $|z|=O(\log t)$.
\end{cor}
\begin{proof}
First, note that the imaginary part of $z$ is bounded and it suffices to assume
that $z$ is real. For simplicity, rescale the initial condition so that $y_0=1$.
Then $x$ and $z$ satisfy the system
\[x'=4+2\frac{1-z}{x-z^2}\gap z'=-\frac{1}{x-z^2}.\]
This implies that
\[0> z'\geq-\frac{1}{4t+x_0-z_0^2},\]
which gives $|z|=O(\log t)$.
\end{proof}
\begin{cor}
Let $\omega(t)$ be a locally homogeneous solution to pluriclosed flow on an
Inoue surface of type $S^+$. Then under the family of metrics
$\frac{\omega(t)}{t}$ the surface $S^+$ converges to a circle of length
$\sqrt{2}|\log\lambda|$ in the Gromov-Hausdorff sense, where $\lambda$ depends
on the construction of $S^+$. For an Inoue surface of type $S^-$ the surface
converges to a circle of length $2\sqrt{2}|\log\lambda|$.
\end{cor}
\begin{proof}
We recall the construction of an Inoue surface of type $S^+$. Let $N\in
SL_2\mathbb{Z}$ have positive eigenvalues $\lambda\neq 1,\lambda^{-1}$ and
corresponding eigenvectors $(a_1,a_2)^T,(b_1,b_2)^T$. Choose integers $j,k,l$
with $l\neq 0$ and a complex number $\kappa$. Let $Sol_1^4$ be the group of
matrices of the form
\[\begin{pmatrix}
1 & u & v\\
0 & q & r\\
0 & 0 & 1
\end{pmatrix}\]
where $q,r,u,v\in\arr$ and $q>0$. For $m\in\arr$ this group has transitive
actions on $\cee H^1\times\cee$ with trivial stabilizers so that an element as
before maps $(\ii ,0)$ to
\[(r+\ii q,v+\ii (u+m\log q),\]
with $Sol_1$ corresponding to the $m=0$ and $Sol_1'$ corresponding to the $m\neq
0$ cases. By taking $m=\frac{\Im (\kappa)}{\log \lambda}$ we can obtain a
cocompact lattice $\Gamma$ generated by
\[g_0=\begin{pmatrix}
1 & 0 & \Re(\kappa)\\
0 & \lambda & 0\\
0 & 0 & 1
\end{pmatrix}\gap g_i=\begin{pmatrix}
1 & b_i & c_i\\
0 & 1 & a_i\\
0 & 0 & 1
\end{pmatrix}\gap g_3=\begin{pmatrix}
1 & 0 & \frac{b_1a_2-b_2a_1}{l}\\
0 & 1 & 0\\
0 & 0 & 1
\end{pmatrix},\]
where $c_i$ are determined by the previously chosen data. The quotient by this
lattice gives an Inoue surface of type $S^+$. An Inoue surface of type $S^{-}$
is formed by a similar quotient where $m=0$ and $\lambda$ is replaced with
$\lambda^2$. As shown by Inoue, these are compact complex surfaces of class
$VII_0$ with no curves and, similarly to the type $S_A$ case, these are bundles
over $S^1$ such that the real and imaginary parts of $Z_2$ span an integrable
distribution whose leaves are dense in each fiber. The curve $\gamma:[0,1]\to
S$ given by $\gamma(t)=(\ii \lambda^t,0)$ provides the relevant Gromov-Hausdorff
approximation for $t$ sufficiently large and we observe that the fibers are
shrinking with respect to the metrics $\frac{\omega(t)}{t}$ as in the case of Inoue surfaces of type $S_A$.
\end{proof}
We can now complete the proof of Theorem \ref{grohau}.
\begin{proof}
Let $\omega(t)$ be a locally homogeneous solution to pluriclosed flow on a compact complex surface which exists on $[0,\infty)$. If $\omega_0$ is K\"ahler then it is a product of K\"ahler-Einstein metrics with non-positive scalar curvatures. Under the rescaled metrics $\frac{\omega(t)}{t}$ the surface either converges to a product of K\"ahler-Einstein metrics with negative scalar curvature, as in the case where the universal covering metric is $\cee H^2$ or $\cee H^1\times\cee H^1$, collapses to a curve of genus $g\geq 2$, as in the case $\cee H^1\times\cee$, or collapses to a point, as in the case of a flat metric on $\cee^2$. If $\omega_0$ is non-K\"ahler, note that it must be a left invariant Hermitian metric on one of the Lie groups considered before. Given that it is not a solution on the Hopf surface, the claimed Gromov-Hausdorff convergence follows from the case by case analysis considered throughout this section.
\end{proof}
\section{Blowdown Limits of Homogeneous Solutions}
Let $(M,g(t))$ be a 1-parameter family of Riemannian manifolds for
$t\in(0,\infty)$. Suppose that a blowdown limit
$g_\infty(t)=\underset{s\rightarrow\infty}\lim s^{-1} g(st)$ exists in the sense
that there is a 1-parameter family of diffeomorphisms $\theta_s$ of $M$ such
that $\theta_{s_1}\circ\theta_{s_2}=\theta_{s_1s_2}$ and
$g_\infty(t)=\underset{s\rightarrow\infty}\lim \theta_s^*s^{-1} g(st)$ uniformly
on compact subsets of $M\times(0,\infty)$. Fix some $a>0$. Then
\[g_\infty(at)=\underset{s\rightarrow\infty}\lim \theta_s^*s^{-1}
g(sat)=a\theta_{a^{-1}}^*\underset{\tilde{s}\rightarrow\infty}\lim
\theta_{\tilde{s}}^*\tilde{s}^{-1}
g(\tilde{s}t)=a\theta_{a^{-1}}^*g_\infty(t).\]
This implies that
\[g_\infty(t)=tg_\infty(1)\]
up to diffeomorphisms of $M$. Therefore, if $g_\infty(t)$ satisfies some
geometric flow then it is an expanding soliton solution of that flow. This
construction has been used in \cite{lott2} to give expanding soliton solutions to Ricci flow
by performing blowdown limits of type III homogeneous Ricci flows.

In this section we will construct expanding soliton solutions to pluriclosed
flow by applying blowdown limits to the homogeneous solutions of the previous
section. 
We will write $X_1$ and $-X_2$ for the real and imaginary parts of $Z_1$
respectively, and similarly for $X_3$, $-X_4$, and $Z_2$. With respect to the
dual basis $\sigma^i$ to the $X_i$ we see that 
\[\zeta^1=\frac{1}{2}(\sigma^{1}+\ii
\sigma^{2})\gap\zeta^2=\frac{1}{2}(\sigma^3+\ii \sigma^4),\] and so our metrics
have the form
\[\omega=\frac{1}{2}(x\sigma^{12}+y\sigma^{34}-\Im(z)(\sigma^{13}+\sigma^{24}
)+\Re(z)(\sigma^{14}-\sigma^{23})).\]

\subsection{The Hyperelliptic Case}
\begin{prop}
Let $\omega(t)$ be a left invariant solution of pluriclosed flow on $\tilde
E(2)\times\arr$. There is a blowdown limit
$\omega_\infty(t)=\lim_{s\to\infty}s^{-1}\omega(st)$ of the form
\[\omega_\infty(t)=\frac{1}{2}(x_0\sigma^{12}+y_0\sigma^{34})\]
which is an expanding soliton solution. It is isometric to the flat metric on
$\cee^2$.
\end{prop}
\begin{proof}
Recall that $x'=y'=0$, and $|z|=O(e^{-Ct})$ for some constant $C$ depending on
the initial conditions.
Define diffeomorphisms $\psi_s:\arr^4\to \tilde E(2)\times\arr$ by
\[\psi_s(q,r,u,v)=\alpha_s(q,r)\beta_s(u,v),\]
where 
\[\alpha_s(q,r)=\exp(\sqrt{s}(qX_1+rX_2))\] 
and \[\beta_s(u,v)=\exp(\sqrt{s}(uX_3+vX_4)).\] 
Write $\theta_s=\psi_s\circ\psi_1^{-1}$. We see that
\[s^{-1}\theta^*_s(\omega(st))=\frac{1}{2}(x_0\sigma^{12}+y_0\sigma^{34})+O(e^{
-Cst})\]
and so there is a blowdown limit
\[\omega_\infty(t)=\frac{1}{2}(x_0\sigma^{12}+y_0\sigma^{34}).\]
\end{proof}
\subsection{The Non-K\"ahler, Properly Elliptic Case}
Recall that our $T^{1,0}$ frame has Lie brackets
\[[Z_1,Z_2]=\ii Z_2\gap [Z_1,\overline{Z_2}]=\ii Z_1\]
\[[Z_1,\overline{Z_1}]=(\ii -\alpha)Z_2+(\ii +\alpha)\overline{Z_2}.\]
With respect to the basis $X_i$ described above, the Lie brackets are
\[[X_1,X_2]=X_3-\alpha X_4\gap [X_3,X_2]=X_1\gap [X_1,X_3]=X_2\]
and the complex structure is given by $JX_1=X_2$ and $JX_3=X_4$.
\begin{lemma}
Any element of $\tilde{SL_2}\arr\times\arr$ can be written uniquely as
$\exp(qX_1+rX_2)\exp(uX_3+vX_4)$.
\end{lemma}
\begin{proof}
This follows from \cite{lott2}, Lemma 3.34, where we note that $X_4$ is central
and our $X_1$, $X_2$, and $X_3$ correspond, respectively, to $X_3$, $X_2$, and
$X_1$ of that Lemma.
\end{proof}
\begin{prop}
Let $\omega(t)$ be a left invariant solution to pluriclosed flow on
$\tilde{SL_2}\arr\times\arr$.
Then there is a blowdown limit $\omega_\infty(t)$ given by a product metric
\[\omega_\infty(t)=\omega_{\cee H^1}(t)\oplus\omega_{\cee}\]
on $\cee H^1\times \cee$ which is an expanding soliton solution.
\end{prop}
\begin{proof}
The argument is the same as in \cite{lott2} for the blowdown limit for a
homogeneous Ricci flow on $\tilde {SL}_2\arr$. Recall that $x\sim 2t$, $y'=0$,
and $|z|=O(e^{-Ct})$. Consider the family of diffeomorphisms
$\psi_s:\arr^4\to\tilde SL_2\arr\times\arr$ given by
\[\psi_s(q,r,u,v)=\exp(qX_1+rX_2)\exp(s^{\frac{1}{2}}(uX_3+vX_4)),\]
write $A(q,r)=\exp(qX_1+rX_2)$ and $B_s(u,v)=\exp(s^{\frac{1}{2}}(uX_3+vX_4))$
and let \[h^{-1}dh=B_s^{-1}A^{-1}dAB_s+s^{\frac{1}{2}}(duX_3+dvX_4)\] be the
Maurer-Cartan form. We compute
\[s^{-1}\psi_s^*\omega(st)\sim\frac{1}{2}(2t((B^{-1}_sA^{-1}dAB_s)_1\wedge(B^{-1
}_sA^{-1}dAB_s)_2)+y_0(du\wedge dv)).\]
The proof is concluded by noting that conjugation by $B_s$ gives a rotation in
the $(q,r)$-plane, so there is a blowdown limit
\[\omega_\infty(t)=\frac{1}{2}(2t (A^{-1}dA)_1\wedge(A^{-1}dA)_2+y_0du\wedge
dv),\]
which is now a product metric solution on $\cee H^1\times\cee$.
\end{proof}
\subsection{The Kodaira Surface Cases}
\begin{prop}
Let $\omega(\cdot)$ be a left invariant solution to pluriclosed flow on
$Nil\times\arr$. Then there is a blowdown limit
\[\omega_\infty(t)=\frac{1}{2}(2\sqrt{y_0t}\sigma^{12}+y_0\sigma^{34})\]
of $\omega(\cdot)$ which is an expanding soliton solution.
\end{prop}
\begin{proof}
Recall that $x\sim 2\sqrt{y_0 t}$ and $y'=z'=0$.
Define diffeomorphisms $\psi_s:\arr^4\to Nil\times\arr$ by
\[\psi_s(q,r,u,v)=\left(\begin{pmatrix} 1 & s^{\frac{1}{4}}r & s^{\frac{1}{2}}u
\\ 0 & 1 & s^{\frac{1}{4}}q \\ 0 & 0 & 1
\end{pmatrix},s^{\frac{1}{2}}v\right),\]
and let $\theta_s=\psi_s\circ\psi_1^{-1}$. We see that
\[s^{-1}\theta_s^*\omega(st)=\frac{1}{2}(2\sqrt{y_0t}\sigma^{12}+y_0\sigma^{34}
)+O(s^{-\frac{1}{4}}).\]
Therefore the blowdown limit is given by
\[\omega_\infty(t)=\frac{1}{2}(2\sqrt{y_0 t}\sigma^{12}+y_0\sigma^{34})\]
\end{proof}
\begin{prop}
Let $\omega(t)$ be a left invariant solution to pluriclosed flow on
$Nil\rtimes\arr$. Then there is a blowdown limit
\[\omega_\infty(t)=\frac{1}{2}(2\sqrt{y_0t}\sigma^{12}+y_0\sigma^{34})\]
given by the expanding soliton solution on $Nil\times\arr$
\end{prop}
\begin{proof}
Recall that $x\sim 2\sqrt{y_0 t}$, $y'=0$, and $|z|=O(e^{-Ct})$. Note that these groups $Nil\rtimes\arr$ and $Nil\times\arr$ are diffeomorphic. For the same coordinates as in the previous case, use the same
diffeomorphisms $\theta_s$ to obtain the desired
blowdown limit.
\end{proof}
\subsection{The Inoue Surface Cases}\
Let $\omega(t)$ be a left invariant solution to pluriclosed flow on the solvable
family. Recall that $x'=0$, $y\sim 12a^2t$, and $|z|=O(1)$.
\begin{prop}
There is a blowdown limit 
\[\omega_\infty(t)=\frac{1}{2}(x_0\sigma^{12}+12a^2t\sigma^{34})\]
\end{prop}
\begin{proof}
Define diffeomorphisms $\psi_s:\arr^4\to G$ by
\[\psi_s(q,r,u,v)=\begin{pmatrix}e^{av}\cos(bv) & -e^{av}\sin(bv) & 0 &
s^{\frac{1}{2}}q\\e^{av}\sin(bv) & e^{av}\cos(bv) & 0 & s^{\frac{1}{2}}r \\ 0 &
0 & e^{-2av} & u \\ 0 & 0 & 0 & 1\end{pmatrix}\]
and let $\theta_s=\psi_s\circ\psi_1^{-1}$. Then
\[s^{-1}\theta_s^*g(st)=\frac{1}{2}(x_0\sigma^{12}+12a^2t\sigma^{34})+O(s^{
-\frac{1}{2}})\]
\end{proof}
\begin{prop}
Let $\omega(t)$ be a left invariant solution to pluriclosed flow on $Sol_1$.
There is a blowdown limit
\[\omega_\infty(t)=\frac{1}{2}(4t\sigma^{12}+y_0\sigma^{34})\]
\end{prop}
\begin{proof}
Recall that $x\sim 4t$, $y'=0$, and $z$ is bounded.
Define diffeomorphisms $\psi_s:\arr^3\times\arr^+\to Sol_1$ by
\[\psi_s(u,v,r,q)=\begin{pmatrix}
1 & s^{\frac{1}{2}}u & s^{\frac{1}{2}}v \\
0 & q & r \\
0 & 0 & 1
\end{pmatrix}\]
and let $\theta_s$ be as before. We see that
\[s^{-1}\theta_s^*g(st)=\frac{1}{2}(4t\sigma^{12}+y_0\sigma^{34})+O(s^{-\frac{1}
{2}})\]
\end{proof}
\begin{prop}
Let $\omega(t)$ be a left invariant solution to pluriclosed flow on $Sol_1'$.
There is a blowdown limit
\[\omega_\infty(t)=\frac{1}{2}(4t\sigma^{12}+y_0\sigma^{34})\]
given by the expanding soliton solution with respect to the other complex
structure $Sol_1$.
\end{prop}
\begin{proof}
Recall that $x\sim 4t$, $y'=0$, and $|z|=O(\log t)$. Let $\sigma^i$ be one forms
dual to the $X_i$ associated to the complex structure $Sol_1$. The one forms
$\tilde\sigma^i$, dual to the $\tilde X_i$ associated to the complex structure
$Sol_1'$, are given by
\[\tilde{\sigma}^i=\sigma^i\gap(i=1,2,3)\gap\tilde \sigma^4=\sigma^4-\sigma^2.\]
If
\[\omega=\frac{1}{2}(x\tilde \sigma^{12}+y\tilde
\sigma^{34}-\Im(z)(\tilde\sigma^{13}+\tilde \sigma^{24})+\Re(z)(\tilde
\sigma^{14}-\tilde \sigma^{23})\]
is a left invariant Hermitian metric with respect to the complex structure
$Sol_1'$, we see that
\[\omega=\frac{1}{2}(x\sigma^{12}+y(\sigma^{34}-\sigma^{32})-\Im(z)(\sigma^{13}
+\sigma^{24})+\Re(z)(\sigma^{14}-\sigma^{12}-\sigma^{23})).\]
Therefore, using the same diffeomorphisms $\theta_s$ as in the previous
proposition, we find a blowdown limit
\[\omega_\infty(t)=\frac{1}{2}(4t\sigma^{12}+y_0\sigma^{34})\]
\end{proof}
We are now able to prove Theorem \ref{soli}
\begin{proof}
Let $\omega(t)$ be a locally homogeneous solution to pluriclosed flow on a compact complex surface $M$ which exists on the maximal time interval $[0,T)$. If $T$ is finite then $M$ must be either $\cee P^2$ or a product of $\cee P^1$ with a curve. Suppose then that $T=\infty$. If $\omega(t)$ is K\"ahler then the induced metric on the universal cover of $M$ is one of $\cee H^2$, $\cee H^1\times \cee H^1$, $\cee H^1\times \cee$, or $\cee ^2$. In each case, the induced metric is a product of K\"ahler-Einstein metrics and it is easy to obtain the required diffeomorphisms to show the existence of a blowdown limit. If $\omega(t)$ is non-K\"ahler, it is either a Hopf surface or one of the metrics considered in this section. It was shown in Corollary 3.5 that the pluriclosed flow of a homogeneous metric on a Hopf surface converges to a canonical metric up to homothety, and we have already constructed blowdown limits for the remaining cases.
\end{proof}
\section{Conclusion}
We have seen the long time behavior of pluriclosed flow on a wide variety non-K\"ahler complex surfaces and have observed that information of the complex structure of such a surface is contained in the asymptotic behavior of the flow. This is a small step toward using the flow to study all non-K\"ahler complex surfaces, and the philosophy is that one would like to take a result which holds for the (K\"ahler-)Ricci flow and prove an analogous result in this pluriclosed setting. 

For example, it was shown in \cite{lott1} that a 3-dimensional type III solution to Ricci flow has a pullback metric on the universal cover which approaches a homogeneous expanding soliton. The main ingredients of the proof of that are a groupoid compactness result for type III Ricci flow solutions and a modified expanding entropy functional $\mathcal W_+$. It is not unreasonable to think that a similar compactness result holds for solutions to pluriclosed flow, and, as shown in Section 6 of \cite{st3}, after pulling back by diffeomorphisms pluriclosed flow admits an analogous monotonic expanding entropy $\mathcal{W}_+$. It is highly likely that one can modify this entropy to prove that type III solutions to pluriclosed flow approach homogeneous metrics when pulled back to the universal cover. Work toward this end is ongoing.

We are also interested in the behavior of the family of flows
\[\frac{d\omega}{dt}=-(\rho^\tau_\omega)^{(1,1)},\]
for different values of $\tau$. Here $\rho_\omega^\tau$ is the Ricci-form associated to the connection $\nabla^\tau$ in the canonical family of Gauduchon \cite{ga1}. The case $\tau=-1$ gives the Bismut connection and corresponds to the pluriclosed flow considered in this paper. The case $\tau=1$ gives the Chern connection and the flow corresponds to the Chern-Ricci flow considered by Tosatti and Weinkove and, similar to what we have done here, homogeneous and soliton solutions to Chern-Ricci flow on Lie groups have been studied in \cite{lau}. We are interested in the bifurcation theory for this family of flows on homogeneous complex surfaces. For example, the Chern-Ricci flow of a homogeneous metric on the Hopf surface must have a finite time singularity, but we have seen that pluriclosed flow always converges to a canonical metric in this case. Additionally, as shown in \cite{lau}, any left invariant Hermititan structure on a nilpotent Lie group is fixed under Chern-Ricci flow, while we have seen non-trivial solutions to the pluriclosed flow on $Nil\times\arr$. Therefore there are values of $\tau$ which induce qualitative changes in the behavior of solutions to this family of flows on a complex surface and, speculatively, the corresponding connections $\nabla^\tau$ may be canonical for the complex surface in some sense.

\bibliographystyle{amsplain}
\bibliography{pluriclosed_refs}

\end{document}